\documentclass[preprint,10pt]{elsarticle}
\usepackage{amsfonts}
 \usepackage{graphics}
\usepackage{graphicx}
 \usepackage{epsfig}

\usepackage{amssymb}
\usepackage{amsthm}

\newtheorem{thm}{Theorem}
\newtheorem{lem}[thm]{Lemma}
\newdefinition{rmk}{Remark}
\newproof{pf}{Proof}
\newdefinition{example}{Example}
\newdefinition{definition}{Definition}
\newdefinition{proposition}{Proposition}
\newdefinition{problem}{Problem}
\newdefinition{corollary}{Corollary}
\journal{Arkiv}

\begin{document}

\begin{frontmatter}



\title{A new error bound for linear complementarity problems for $B$-matrices}

\author%[rvt]
{Chaoqian Li \corref{cor1}} \ead{lichaoqian@ynu.edu.cn}

\author%[rvt]
{Mengting Gan} %\ead{jaq1029@163.com}
\author%[rvt]
{Shaorong Yang}%\corref{cor1}}
%\ead{liyaotang@ynu.edu.cn}

\cortext[cor1]{Corresponding author.}

%\address[focal]{Jilin Vocational college of Industry and
%Technology, Jilin,  P. R. China 132000}

\address%[rvt]
{School of Mathematics and Statistics, Yunnan
University, Yunnan,  P. R. China 650091}

\begin{abstract}
A new error bound for the linear complementarity problem is given
when the involved matrix is a $B$-matrix. It is shown that this
bound is sharper than some previous bounds [C.Q. Li, Y.T. Li. Note
on error bounds for linear complementarity problems for
$B$-matrices,\newblock {\em Applied Mathematics Letters,}
57:108-113, 2016] and [C.Q. Li, Y.T. Li. Weakly chained diagonally
dominant $B$-matrices and error bounds for linear complementarity
problems, to appear in Numer. Algor.].
%those provided by
%Garc\'{i}a-Esnaola and Pe\~{n}a ([Error bounds for linear
%complementarity problems for $B$-matrices. Appl. Math. Lett.,
%22:1071-1075, 2009]), Li and Li ([Note on error bounds for linear
%complementarity problems for $B$-matrices, Apllied Mathematics
%Letter],[Weakly chained diagonally dominant $B$-matrices and error
%bounds for linear complementarity problems, Numer. Algor.]).
\end{abstract}


\end{frontmatter}

% \linenumbers

\section{Introduction}
 Linear complementarity problem  LCP$(M, q)$ is to find
a vector $x\in R^{n}$ such that
\begin{equation}\label{neq1.1}
x\geq 0, Mx+q\geq 0, (Mx+q)^Tx=0, \end{equation} where $M
=[m_{ij}]\in R^{n\times n}$ and $q\in R^n$. The $ LCP(M, q)$  has
various applications in the Nash equilibrium point of a bimatrix
game, the contact problem and the free boundary problem  for journal
bearing, for details,  see \cite{Chen,Co,Mu}.

It is well-known that the LCP$(M, q)$ has a unique solution for any
$q\in R^n$ if and only if $M$ is a $P$-matrix \cite{Co}. Here, a
matrix $M \in R^{n\times n}$ is called a $P$-matrix if  all its
principal minors are positive \cite{Pena1}. In \cite{Chen0}, Chen
and Xiang gave the following error bound of the LCP$(M, q)$ when $M$
is a $P$-matrix:
\[ ||x-x^*||_{\infty} \leq \max\limits_{d\in [0,1]^n}||(I -D+DM)^{-1}||_{\infty} ||r(x)||_\infty,\]
where $x^*$ is the solution of the LCP$(M, q)$, $r(x)=\min\{
x,Mx+q\}$, $D=diag(d_i)$ with $0\leq d_i \leq 1$, and the min
operator $r(x)$ denotes the componentwise minimum of two vectors. If
$M$ satisfies  certain structure, then some  bounds of
$\max\limits_{d\in [0,1]^n}||(I -D+DM)^{-1}||_{\infty}$ can be
derived; for details, see \cite{Che,Dai,Dai1,Ga,Ga4} and references
therein.

When $M$ is a $B$-matrix introduced by Pe\~{n}a in \cite{Pena1} as a
subclass of $P$-matrices, Garc\'{\i}a-Esnaola and Pe\~{n}a in
\cite{Ga} presented the following upper bound which is only related
with the entries of $M$. Here a real matrix $M=[m_{ij}]\in
R^{n\times n}$ is called a $B$-matrix  \cite{Pena1} if for each
$i\in N=\{1,2,\ldots, n\}$,
\begin{equation}\label{neq1.2} \sum\limits_{k\in N} m_{ik}>0,~and~
\frac{1}{n} \left(\sum\limits_{k\in N} m_{ik} \right)>
m_{ij},~for~any~ j\in N~ and~j\neq i .\end{equation}

\begin{thm} \emph{\cite[Theorem 2.2]{Ga} }\label{th1.1} Let $M=[m_{ij}]\in R^{n\times n}$ be a $B$-matrix with the form
\begin{equation}\label{dec} M =B^++C,\end{equation} where
\begin{equation}  \label{neq1.3} B^+ =[b_{ij}]= \left[ \begin{array}{ccc}
 m_{11}-r_1^+    &\cdots     &m_{1n}-r_1^+ \\
 \vdots          &          &\vdots  \\
 m_{n1}-r_n^+     &\cdots    &m_{nn}-r_n^+
\end{array} \right], \end{equation}
and $r_i^+=\max\{ 0,m_{ij}|j\neq i\}$. Then
\begin{equation} \label{neq1.4} \max\limits_{d\in [0,1]^n}||(I
-D+DM)^{-1}||_{\infty}\leq
\frac{n-1}{\min\{\beta,1\}},\end{equation} where
$\beta=\min\limits_{i\in N}\{\beta_i\}$ and
$\beta_i=b_{ii}-\sum\limits_{j\neq i}|b_{ij}|$.
\end{thm}

As shown in \cite{Li}, if the diagonal dominance of $B^+$ is weak,
i.e.,
\[\beta=\min\limits_{i\in N}\{\beta_i\}=\min\limits_{i\in
N}\left\{b_{ii}-\sum\limits_{j\neq i}|b_{ij}|\right\}\] is small,
 then the bound (\ref{neq1.4})  may be very large when $M$ is
a $B$-matrix, which leads to that the estimate  in (\ref{neq1.4}) is
always inaccurate, for details, see \cite{Li,Li1}. To improve the
bound (\ref{neq1.4}), Li and Li \cite{Li} gave the following  bound
for $\max\limits_{d\in [0,1]^n}||(I -D+DM)^{-1}||_{\infty}$ when $M$
is a $B$-matrix.

\begin{thm}\emph{\cite[Theorem 4]{Li} }   \label{th1.2} Let $M=[m_{ij}]\in R^{n\times n}$ be a $B$-matrix with the
form $ M=B^++C$, where $B^+=[b_{ij}]$ is the matrix of
(\ref{neq1.3}). Then
\begin{equation} \label{neq1.5}  \max\limits_{d\in [0,1]^n}||(I
-D+DM)^{-1}||_{\infty} \leq\sum\limits_{i=1}^n
\frac{n-1}{\min\left\{ \bar{\beta}_i, 1\right\}}
\prod\limits_{j=1}^{i-1} \left(1+ \frac{1}{\bar{\beta}_j}
\sum\limits_{k=j+1}^n|b_{jk}| \right),\end{equation} where
$\bar{\beta}_i=b_{ii}-\sum\limits_{j=i+1}^n|b_{ij}|l_i(B^+) $,
$l_k(B^+)= \max\limits_{k\leq i\leq n}\left\{ \frac{1}{|b_{ii}|}
\sum\limits_{j=k,\atop j\neq i}^n |b_{ij}| \right\} $ and
\[\prod\limits_{j=1}^{i-1} \left(1+ \frac{1}{\bar{\beta}_j}
\sum\limits_{k=j+1}^n|b_{jk}| \right) =1 ~if~ i=1.\]
\end{thm}

Very recently, when $M$ is a weakly chained diagonally dominant
$B$-matrix,  Li and Li \cite{Li1} gave a bound for
$\max\limits_{d\in [0,1]^n}||(I -D+DM)^{-1}||_{\infty}$. This bound
holds true for the case that $M$ is a $B$-matrix because a
$B$-matrix is a weakly chained diagonally dominant $B$-matrix
\cite{Li1}.

\begin{thm} \emph{\cite[Corollary 1]{Li1} } \label{th1.3}
Let $M=[m_{ij}]\in R^{n\times n}$ be a $B$-matrix with the form $
M=B^++C$, where $B^+=[b_{ij}]$ is the matrix of (\ref{neq1.3}). Then
\[\max\limits_{d\in [0,1]^n}||(I
-D+DM)^{-1}||_{\infty} \leq \sum\limits_{i=1}^{n}
\left(\frac{n-1}{\min\{\tilde{\beta}_i,1\}} \prod \limits_{j =1}^{
i-1} \frac{b_{jj}}{\tilde{\beta}_j} \right),\] where
$\tilde{\beta}_i=b_{ii}-\sum\limits_{j= i+1}^{n}|b_{ij}|
>0$ and $ \prod \limits_{j =1}^{
i-1} \frac{b_{jj}}{\tilde{\beta}_j} =1$ if $i=1.$
\end{thm}

In this paper, we also focus on the  error bound for the LCP$(M,
q)$, and gave a new bound for  $\max\limits_{d\in [0,1]^n}||(I
-D+DM)^{-1}||_{\infty}$ when $M$ is a $B$-matrix. It is shown that
this bound is more effective to estimate  $\max\limits_{d\in
[0,1]^n}||(I -D+DM)^{-1}||_{\infty}$ than that in Theorem
\ref{th1.1}, and sharper than those in Theorems \ref{th1.2} and
\ref{th1.3}.

%%%%%%%%%%%%%%%%%%%%%%%%%%%%%%%%%%%%%%%%%%%%%%%%%%%%%%%%%%%%%%%

%%%%%%%%%%%%%%%%%%%%%%%%%%%%%%%%%%%%%%%%%%%%%%%%%%%
\section{Main results}\vspace{-2pt}
We first recall some definitions. A matrix $A=[a_{ij}]\in C^{n\times
n}$ is called a strictly diagonally dominant ($SDD$) matrix if for
each $i\in N$, $ |a_{ii}|> \sum\limits_{j=1,\atop j\neq i}^n
|a_{ij}|$. It is well-known that an $SDD$ matrix is nonsingular
\cite{Ber}. A matrix $A=[a_{ij}]$ is called a $Z$-matrix if
$a_{ij}\leq 0$ for any $i\neq j$, and a nonsingular $M$-matrix if
$A$ is a $Z$-matrix with $A^{-1}$ being nonnegative \cite{Ber}.
Next, several lemmas which will be used later are given.

\begin{lem}\cite[Theorem 3.2]{Wang} \label{nle2.1}
Let $A=[a_{ij}]\in R^{n\times n}$ be an $SDD$ $M$-matrix. Then
\[||A^{-1}||_\infty \leq \sum\limits_{i=1}^n\left(\frac{1}{a_{ii}(1-u_i(A)l_i(A))}
 \prod\limits_{j=1}^{i-1} \frac{1}{1-u_j(A)l_j(A)}\right)
 ,\]
where $u_i(A)=\frac{1}{|a_{ii}|}  \sum\limits_{j=i+1}^n |a_{ij}|$,
$l_k(A)=\max\limits_{k\leq i\leq n}\left\{ \frac{1}{|a_{ii}|}
\sum\limits_{j=k,\atop j\neq i}^n |a_{ij}| \right\}$ and
\[\prod\limits_{j=1}^{i-1} \frac{1}{1-u_j(A)l_j(A)}=1 ~if ~i=1.\]
\end{lem}

\begin{lem} \cite[Lemma 3]{Li} \label{nle2.2}
Let $\gamma > 0$ and $ \eta \geq 0 $. Then for any $x\in [0,1]$,
\begin{equation} \label{neq2.1} \frac{1}{1-x+\gamma x} \leq  \frac{1}{\min\{\gamma,1\}}\end{equation} and
\begin{equation} \label{neq2.2} \frac{\eta x}{1-x+\gamma x} \leq  \frac{\eta }{\gamma}.\end{equation}
\end{lem}

\begin{lem} \cite[Lemma 5]{Li1} \label{nle2.3}
Let $A= [a_{ij}]\in R^{n\times n}$ with \[a_{ii}>\sum\limits_{j=
i+1}^{n}|a_{ij}| ~ for ~each ~i\in N.\] Then for any $x_i\in [0,1]$,
$i\in N$, \[\frac{1-x_i+a_{ii}x_i}{1-x_i+a_{ii}x_i-\sum\limits_{j=
i+1}^{n}|a_{ij}|x_i} \leq  \frac{a_{ii}}{  a_{ii}-\sum\limits_{j=
i+1}^{n}|a_{ij}|}.\]
\end{lem}

The main result is given as follows.

\begin{thm} \label{main} Let $M=[m_{ij}]\in R^{n\times n}$ be a $B$-matrix with the
form $ M=B^++C$, where $B^+=[b_{ij}]$ is the matrix of
(\ref{neq1.3}). Then
\begin{equation} \label{neq2.0}  \max\limits_{d\in [0,1]^n}||(I
-D+DM)^{-1}||_{\infty} \leq \sum\limits_{i=1}^n
\frac{n-1}{\min\left\{ \bar{\beta}_i, 1\right\}}
\prod\limits_{j=1}^{i-1} \frac{b_{jj}}{\bar{\beta}_j}
,\end{equation} where $\bar{\beta}_i $ is defined in Theorem
\ref{th1.2} and $\prod\limits_{j=1}^{i-1}
\frac{b_{jj}}{\bar{\beta}_j}=1$ if $i=1$.
\end{thm}

\begin{proof} Let $M_D=I-D+DM$. Then
\[M_D=I-D+DM=I-D+D(B^++C)=B_D^++C_D,\] where
$B_D^+=I-D+DB^+=[b_{ij}]$. Similar to the proof of Theorem 2.2 in
\cite{Ga}, we can obtain that $B_D^+$ is an $SDD$ $M$-matrix with
positive diagonal elements  and $C_D=DC$, and that
\begin{equation} \label{neq2.3}||M_D^{-1}||_\infty \leq ||
(I+(B_D^+)^{-1}C_D)^{-1}||_\infty ||( B_D^+)^{-1}||_\infty \leq
(n-1)||( B_D^+)^{-1}||_\infty.\end{equation} By Lemma \ref{nle2.1},
\begin{eqnarray} \label{neq2.4}||( B_D^+)^{-1}||_\infty &\leq & \sum\limits_{i=1}^n \left(
\frac{1}{(1-d_i+b_{ii}d_i)\left(1- u_i(B_D^+) l_i(B_D^+)\right)}
\prod\limits_{j=1}^{i-1} \frac{1}{1-u_j(B_D^+)l_j(B_D^+)}
\right)\nonumber,
\end{eqnarray}
where
\[u_i(B_D^+)= \frac{\sum\limits_{j=i+1}^n |b_{ij}|d_i}{1-d_i+b_{ii}d_i} ,
and~ l_k(B_D^+)= \max\limits_{k\leq i\leq n}\left\{
\frac{\sum\limits_{j=k,\atop j\neq i}^n
|b_{ij}|d_i}{1-d_i+b_{ii}d_i}
 \right\}. \]
By Lemma \ref{nle2.2}, we can easily get that for each $k\in N$,
\begin{equation} \label{neq2.5}l_k(B_D^+) \leq  \max\limits_{k\leq i\leq n}\left\{
\frac{1}{b_{ii}}\sum\limits_{j=k,\atop j\neq i}^n |b_{ij}| \right\}
= l_k(B^+) < 1, \end{equation} and that for each $i\in N$,
\begin{eqnarray}\label{neq2.6} \nonumber \frac{1}{(1-d_i+b_{ii}d_i)\left(1- u_i(B_D^+)
l_i(B_D^+)\right)}&=&\frac{1}{1-d_i+b_{ii}d_i- \sum\limits_{j=i+1}^n
|b_{ij}|d_i l_i(B_D^+)} \\& \leq &\frac{1}{\min\left\{b_{ii}-
\sum\limits_{j=i+1}^n|b_{ij}| l_i(B^+), 1\right\}}\nonumber\\&=&
\frac{1}{\min\left\{ \bar{\beta}_i, 1\right\}}. \end{eqnarray}
Furthermore, by Lemma \ref{nle2.3},
\begin{equation} \label{neq2.7}\frac{1}{1-u_i(B_D^+)l_i(B_D^+)}=\frac{1-d_i+b_{ii}d_i}{1-d_i+b_{ii}d_i- \sum\limits_{j=i+1}^n |b_{ij}|d_i
l_i(B_D^+)} \leq \frac{b_{ii}}{\bar{\beta}_i}
 .\end{equation} By (\ref{neq2.4}),
(\ref{neq2.5}), (\ref{neq2.6}) and (\ref{neq2.7}), we have
\begin{equation} \label{neq2.8}||( B_D^+)^{-1}||_\infty \leq
\frac{1}{\min\left\{ \bar{\beta}_1, 1\right\}}+ \sum\limits_{i=2}^n
\frac{1}{\min\left\{ \bar{\beta}_i, 1\right\}}
\prod\limits_{j=1}^{i-1}  \frac{b_{jj}}{\bar{\beta}_j} .
\end{equation}
The conclusion follows from (\ref{neq2.3}) and (\ref{neq2.8}).
\end{proof}

The comparisons of the bounds in Theorems \ref{th1.2}, \ref{th1.3}
and \ref{main} are established as follows.

\begin{thm}  Let $M=[m_{ij}]\in R^{n\times n}$ be a $B$-matrix with the
form $ M=B^++C$, where $B^+=[b_{ij}]$ is the matrix of
(\ref{neq1.3}). Let $\bar{\beta}_i $ and  $\tilde{\beta}_i$ be
defined in Theorems \ref{th1.2} and \ref{th1.3} respectively.  Then
\begin{eqnarray*}
 \sum\limits_{i=1}^n \frac{n-1}{\min\left\{
\bar{\beta}_i, 1\right\}} \prod\limits_{j=1}^{i-1}
\frac{b_{jj}}{\bar{\beta}_j}&\leq &\sum\limits_{i=1}^n
\frac{n-1}{\min\left\{ \bar{\beta}_i, 1\right\}}
\prod\limits_{j=1}^{i-1} \left(1+ \frac{1}{\bar{\beta}_j}
\sum\limits_{k=j+1}^n|b_{jk}| \right)\nonumber
\\&\leq&\sum\limits_{i=1}^{n}
\left(\frac{n-1}{\min\{\tilde{\beta}_i,1\}} \prod \limits_{j =1}^{
i-1} \frac{b_{jj}}{\tilde{\beta}_j} \right).
\end{eqnarray*}
\end{thm}

\begin{proof}Note that
\[\tilde{\beta}_i=b_{ii}-\sum\limits_{j= i+1}^{n}|b_{ij}|,~\bar{\beta}_i=b_{ii}-\sum\limits_{j=i+1}^n|b_{ij}|l_i(B^+) \] and
$l_k(B^+)= \max\limits_{k\leq i\leq n}\left\{ \frac{1}{|b_{ii}|}
\sum\limits_{j=k,\atop j\neq i}^n |b_{ij}| \right\} < 1 $. Hence,
for each $i\in N$,  $\tilde{\beta}_i \leq \bar{\beta}_i $ and
\begin{equation}  \label{neq2.9} \frac{1}{\min\{ \tilde{\beta}_i,1 \}} \geq \frac{1}{\min\{
\bar{\beta}_i,1 \}}. \end{equation} Meantime, for
$j=1,2,\ldots,n-1$,
\begin{equation} \label{neq2.10}   1+ \frac{1}{\bar{\beta}_j}
\sum\limits_{k=j+1}^n|b_{jk}| \leq 1+ \frac{1}{\tilde{\beta}_j}
\sum\limits_{k=j+1}^n|b_{jk}| =
\frac{1}{\tilde{\beta}_j}\left(\tilde{\beta}_j+
\sum\limits_{k=j+1}^n|b_{jk}| \right)=
\frac{b_{jj}}{\tilde{\beta}_j}.
\end{equation}
By (\ref{neq2.9}) and   (\ref{neq2.10}), we have
\begin{equation} \label{neq2.11} \sum\limits_{i=1}^n \frac{n-1}{\min\left\{
\bar{\beta}_i, 1\right\}} \prod\limits_{j=1}^{i-1} \left(1+
\frac{1}{\bar{\beta}_j} \sum\limits_{k=j+1}^n|b_{jk}| \right)
\leq\sum\limits_{i=1}^{n}
\left(\frac{n-1}{\min\{\tilde{\beta}_i,1\}} \prod \limits_{j =1}^{
i-1} \frac{b_{jj}}{\tilde{\beta}_j} \right).\end{equation} Moreover,
for $j=1,2,\ldots,n-1$,
\begin{eqnarray*}
\frac{b_{jj}}{\bar{\beta}_j}&=& \prod\limits_{j=1}^{i-1}
\frac{b_{jj}-\sum\limits_{k=j+1}^n|b_{jk}|l_j(B^+)+\sum\limits_{k=j+1}^n|b_{jk}|l_j(B^+)}{\bar{\beta}_j}\nonumber
\\&=&\frac{\bar{\beta}_j+\sum\limits_{k=j+1}^n|b_{jk}|l_j(B^+)}{\bar{\beta}_j}\nonumber
\\&=&\left(1+\frac{\sum\limits_{k=j+1}^n|b_{jk}|l_j(B^+)}{\bar{\beta}_j}\right)\nonumber
\\&\leq&\left(1+\frac{\sum\limits_{k=j+1}^n|b_{jk}|}{\bar{\beta}_j}\right),
\end{eqnarray*}
this implies \begin{equation} \label{neq2.12} \sum\limits_{i=1}^n
\frac{n-1}{\min\left\{ \bar{\beta}_i, 1\right\}}
\prod\limits_{j=1}^{i-1} \frac{b_{jj}}{\bar{\beta}_j} \leq
\sum\limits_{i=1}^n \frac{n-1}{\min\left\{ \bar{\beta}_i, 1\right\}}
\prod\limits_{j=1}^{i-1} \left(1+ \frac{1}{\bar{\beta}_j}
\sum\limits_{k=j+1}^n|b_{jk}| \right)\nonumber .\end{equation} The
conclusion follows from (\ref{neq2.11})and (\ref{neq2.12}).
\end{proof}

%\begin{rmk}In Corollary 1 of \cite{Li}, Li and Li provided another upper bound for $\max\limits_{d\in [0,1]^n}||(I
%-D+DM)^{-1}||_{\infty}$ where $M$ is a $B$-matrix:
%\begin{equation}\label{neq2.9}\max\limits_{d\in [0,1]^n}||(I
%-D+DM)^{-1}||_{\infty} \leq \sum\limits_{i=1}^{n}
%\left(\frac{n-1}{\min\{\tilde{\beta}_i,1\}} \prod \limits_{j =1}^{
%i-1} \frac{b_{jj}}{\tilde{\beta}_j} \right),\end{equation} where
%\[\tilde{\beta}_i=b_{ii}-\sum\limits_{j= i+1}^{n}|b_{ij}| >0.\]
%Since $l_j(B^+)< 1$ for each  $j\in N$, we have that $\bar{\beta}_j
%> \tilde{\beta}_j $ for $j\in N$, and
% Hence, the bound in (\ref{neq2.0}) is less than that in
%(\ref{neq2.9}), that is,
%\[\frac{n-1}{\min\left\{ \bar{\beta}_1, 1\right\}}+
%\sum\limits_{i=2}^n \frac{n-1}{\min\left\{ \bar{\beta}_i, 1\right\}}
%\prod\limits_{j=1}^{i-1} \left(1+ \frac{1}{\bar{\beta}_j}
%\sum\limits_{k=j+1}^n|b_{jk}|  \right) < \sum\limits_{i=1}^{n}
%\left(\frac{n-1}{\min\{\tilde{\beta}_i,1\}} \prod \limits_{j =1}^{
%i-1} \frac{b_{jj}}{\tilde{\beta}_j} \right).\]
%\end{rmk}

\begin{example} Consider the family of $B$-matrices in \cite{Li}:
\[ M_k=
\left[\begin{array}{cccc}
1.5   & 0.5 &0.4&0.5\\
-0.1   &1.7  &0.7&0.6\\
0.8   & -0.1\frac{ k}{k+1} &1.8 & 0.7 \\
0   &0.7  &0.8&   1.8
\end{array} \right],\]
where $k\geq 1$. Then $M_k=B_k^++C_k$, where
\[ B_k^+=
\left[\begin{array}{cccc}
1   & 0 &-0.1&0\\
-0.8   &1  &0&-0.1\\
0   & -0.1\frac{ k}{k+1}-0.8 &1 & -0.1 \\
-0.8   &-0.1  &0&   1
\end{array} \right].\]
By Theorem \ref{th1.1} (Theorem 2.2 in \cite{Ga}), we have
\[ \max\limits_{d\in [0,1]^4}||(I -D+DM)^{-1}||_{\infty}\leq
\frac{4-1}{\min\{\beta,1\}}= 30(k+1).\] It is obvious that
\[30(k+1)\rightarrow +\infty, when ~k\rightarrow +\infty.\]
By Theorem \ref{th1.3}, we have
\[ \max\limits_{d\in [0,1]^4}||(I -D+DM)^{-1}||_{\infty}\leq
 \sum\limits_{i=1}^{4}
\left(\frac{3}{\min\{\tilde{\beta}_i,1\}} \prod \limits_{j =1}^{
i-1} \frac{b_{jj}}{\tilde{\beta}_j} \right)\thickapprox 15.2675.\]
By Theorem \ref{th1.2}, we have
\[ \max\limits_{d\in [0,1]^4}||(I -D+DM_k)^{-1}||_{\infty} \leq
\frac{2.97(90k+91)(190k+192)+6.24(100k+101)^2}{0.99\left(  90k+91
\right)^2},\] and
\[\frac{2.97(90k+91)(190k+192)+6.24(100k+101)^2}{0.99\left(  90k+91
\right)^2}<15.2675,~for ~any ~k\geq 1.\] By Theorem \ref{main}, we
have
\[ \max\limits_{d\in [0,1]^4}||(I -D+DM_k)^{-1}||_{\infty} \leq
\frac{2.97(90k+91)(190k+191)+5.97(100k+100)^2}{0.99\left(  90k+91
\right)^2},\] and
\begin{eqnarray*}&&\frac{2.97(90k+91)(190k+191)+5.97(100k+100)^2}{0.99\left(  90k+91
\right)^2}\\&<&\frac{2.97(90k+91)(190k+192)+6.24(100k+101)^2}{0.99\left(
90k+91 \right)^2}.\end{eqnarray*} In particular, when $k=1$,
\[\frac{2.97(90k+91)(190k+191)+5.97(100k+100)^2}{0.99\left(  90k+91
\right)^2}\thickapprox 13.6777,\] and
\[\frac{2.97(90k+91)(190k+192)+6.24(100k+101)^2}{0.99\left(  90k+91
\right)^2} \thickapprox 14.1044.\] When $k=2$,
\[\frac{2.97(90k+91)(190k+191)+5.97(100k+100)^2}{0.99\left(  90k+91
\right)^2}\thickapprox  13.7110,\] and
\[\frac{2.97(90k+91)(190k+192)+6.24(100k+101)^2}{0.99\left(  90k+91
\right)^2} \thickapprox  14.1079.\]  In these two cases, the bounds
in (\ref{neq1.4}) are equal to $60$ ($k=1$) and $90$ ($k=2$)
respectively. This example shows that the bound in Theorem
\ref{main} is sharper than those in Theorems \ref{th1.1},
\ref{th1.2} and \ref{th1.3}.
\end{example}

%%%%%%%%%%%%%%%%%%%%%%%%%%%%%%%%%%%%%%%%%%%%%%%%%%%%%%%%%%%%%%%%%%%%%%
\section{Conclusions} In this paper, we give a new bound for $\max\limits_{d\in
[0,1]^n}||(I -D+DM)^{-1}||_{\infty}$ when $M$ is a $B$-matrix, which
improves the bounds obtained  in \cite{Ga,Li,Li1}.

%%%%%%%%%%%%%%%%%%%%%%%%%%%%%%%%%%%%%%%%%%%%%%%%%%%%%%%%%%%%%%%%%%%%%%%

\section*{Acknowledgements}
This work is supported by National Natural Science Foundations of
China (11361074), Natural Science Foundations of Yunnan Province
(2013FD002) and IRTSTYN.


%% References


%\bibliographystyle{elsarticle-num}



\end{document}